\documentclass[a4paper]{scrartcl}
\usepackage[latin1]{inputenc}
\usepackage{amsmath}
\usepackage{bbm}
\usepackage{amssymb}
\usepackage{amsthm}
\usepackage{pdfsync}
\usepackage{esint}
\usepackage{graphicx}
\newtheorem{lemma}{Lemma}[section]
\newtheorem{theorem}{Theorem}[section]

\theoremstyle{remark}
\newtheorem{remark}{Remark}[section]
\DeclareMathOperator{\grad}{grad}
\DeclareMathOperator{\Hess}{Hess}

\numberwithin{equation}{section}

\begin{document}
\title{Plancherel-Rotach formulae for average characteristic polynomials of products of Ginibre random matrices and the Fuss-Catalan distribution}
\author{Thorsten Neuschel%
  \thanks{Department of Mathematics, KU Leuven, Celestijnenlaan 200B box 2400, BE-3001 Leuven, Belgium. This work is supported by KU Leuven research grant OT\slash12\slash073. E-mail: Thorsten.Neuschel@wis.kuleuven.be}}

 \date{\today}

\maketitle

\paragraph{Abstract} Formulae of Plancherel-Rotach type are established for the average characteristic polynomials of certain Hermitian products of rectangular Ginibre random matrices on the region of zeros. These polynomials form a general class of multiple orthogonal hypergeometric polynomials generalizing the classical Laguerre polynomials. The proofs are based on a multivariate version of
the complex method of saddle points. After suitable rescaling the asymptotic zero distributions for the polynomials are studied and shown to coincide with the Fuss-Catalan distributions. Moreover, introducing appropriate coordinates, elementary and explicit characterizations are derived for the densities as well as for the distribution functions of the Fuss-Catalan distributions of general order.


\section{Introduction}

A recently developed field of research in the rich theory of random matrices is the study of eigenvalue distributions for products of random matrices with a fixed number of factors (see, e.g., \cite{Akemann1}, \cite{Akemann2}, \cite{Burda1}, \cite{Burda2}, \cite{Rourke}, \cite{KuijlaarsZhang}). 
Let \(r\in\mathbb{N}=\{1, 2, 3, \ldots\}\) be an arbitrary positive integer and denote by \(X_1, X_2, \ldots, X_r\) complex random matrices such that all entries are independent random variables with a complex Gaussian distribution (matrices of this kind are called Ginibre random matrices). Moreover, let each matrix \(X_j\) be of dimension \(N_j \times N_{j-1}\) and let the matrix \(Y_r\) be defined as the product
\[Y_r=X_r X_{r-1} \cdots X_1.\]
Assuming \(N_0 = \min\{N_0, \ldots, N_r\}\) and writing \(n=N_0\), let us consider the \(n \times n\)-dimensional matrix \(Y_r^{\ast} Y_r\), where \(Y_r^{\ast}\) denotes the conjugate transpose of \(Y_r\). Akemann, Ipsen and Kieburg showed in \cite{Akemann2} that the eigenvalues of \(Y_r^{\ast} Y_r\) are distributed according to a determinantal point process with a correlation kernel expressible in terms of Meijer G-functions. Kuijlaars and Zhang recently showed in \cite{KuijlaarsZhang} that this point process can be interpreted as a multiple orthogonal polynomial ensemble. In this paper we are interested in the average characteristic polynomials of the matrices \(Y_r^{\ast} Y_r\) which are given as generalized hypergeometric polynomials of the form
\begin{equation}\label{DefP}P_n (x) = (-1)^n \prod_{l=1}^{r} (\nu_l +1)_n  ~ _{1} F_{r} \left(\begin{matrix} 
 & -n& \\ \nu_1 +1, & \ldots, &\nu_r +1 & \end{matrix}\,\bigg\vert\,  x \right),
\end{equation}
where \(\nu_j = N_j -N_0\) for \(j\in\{1,\ldots,r\}\) (see \cite{KuijlaarsZhang}, \cite{Akemann2}). These polynomials can be considered as a generalization of the classical Laguerre polynomials, as in the special case \(r=1\) and \(\nu_1 =0\) we have
\[P_n (x) = (-1)^n n! ~ _{1} F_{1} \left(\begin{matrix} 
  -n \\ 1 \end{matrix}\,\bigg\vert\,  x \right)=(-1)^n n! L_{n}^{(0)}(x).\]
The classical Laguerre polynomials satisfy an orthogonality relation on the real line. In the case of general parameters the polynomials \(P_n\) turn out to be multiple orthogonal polynomials of type II with respect to \(r\) weight functions (see \cite{KuijlaarsZhang}).
The origin of those polynomials lies in the theory of random matrices and as there typically exists a close connection between the behavior of the eigenvalues and the roots of the average characteristic polynomials, we are interested in the behavior of the sequence of polynomials \((P_n)_n\) on the asymptotic region of zeros after introducing a suitable rescaling of the argument. More precisely, our aim is to study the asymptotic behavior of the sequence \(P_n (n^r x)\) on the interval \(\left(0,\frac{(r+1)^{r+1}}{r^r}\right)\). From the definition (\ref{DefP}) we immediately obtain the sum representation
\[P_n (x)=(-1)^n \prod_{j=1}^r (n+\nu_j)!~ F_n (x),\]
where we introduce the polynomials
\begin{equation}\label{DefF}F_n (x)= \sum_{k=0}^n \binom{n}{k}\frac{(-x)^k}{(k+\nu_1)! \ldots(k+\nu_r)!}.
\end{equation}
The main result in Theorem \ref{PRA} is an oscillatory Plancherel-Rotach type asymptotic formula for the polynomials \(F_n\) of the following form: Let \(r\in\mathbb{N}\) and \(\nu_1, \ldots, \nu_r \in \mathbb{N}_0\) be arbitrary integers. Then for the polynomials \(F_n\) defined in (\ref{DefF}) we have 
\begin{align} \nonumber F_n (n^r x)=&\frac{2}{(2\pi)^{r/2}} \left(\frac{\sin{r\varphi}}{n \sin{(r+1)\varphi}}\right)^{\frac{r}{2}+\nu_1+\ldots+\nu_r} \exp\left\{nr\frac{\sin{(r+1)\varphi}}{\sin{r\varphi}} \cos{\varphi}\right\}\\ \label{PR1}
\times&\left(-\frac{\sin{r\varphi}}{\sin{\varphi}}\right)^n \left\{\left(1-\frac{r\sin{\varphi}\cos{(r+1)\varphi}}{\sin{r\varphi}}\right)^2+\left(\frac{r\sin{\varphi}\sin{(r+1)\varphi}}{\sin{r\varphi}}\right)^2\right\}^{-1/4}\\\nonumber
\times&\left\{\cos{\left(n\left(r\frac{\sin{(r+1)\varphi}}{\sin{r\varphi}}\sin{\varphi}-(r+1)\varphi\right)+g(r,\nu,\varphi)\right)}+o(1)\right\},
\end{align}
as \(n\rightarrow \infty\). In (\ref{PR1}) we use
\[x=\frac{\left(\sin{(r+1)\varphi}\right)^{r+1}}{\sin{\varphi}\left(\sin{r\varphi}\right)^r},\quad 0<\varphi<\frac{\pi}{r+1},\] 
parameterizing the interval \(\left(0,\frac{(r+1)^{r+1}}{r^r}\right)\), and the phase shift \(g(r,\nu,\varphi)\) is given by
\begin{equation*}g(r,\nu,\varphi)=-\left(\frac{r}{2}+\nu_1+\ldots+\nu_r\right)\varphi-\frac{1}{2}\arctan{\left(\frac{-r\sin{\varphi}\sin{(r+1)\varphi}}{\sin{r\varphi}-r\sin{\varphi \cos{(r+1)\varphi}}}\right)}.
\end{equation*}
The proof of the Plancherel-Rotach type formula (\ref{PR1}) is based on an application of a multivariate version of the method of saddle points as shown in \cite{Neuschel} (see also \cite{Gamkrelidze}, p. 124, and \cite{Fedoryuk}), which we recall in Theorem \ref{MSP} for convenience. As a consequence of Theorem \ref{PRA}, we study the behavior of the zeros in Theorem \ref{WA}. We show that the asymptotic zero distribution of the rescaled polynomials \(F_n(n^r x)\) is given by the Fuss-Catalan distribution of order \(r\) (see, e.g., \cite{Penson}). Furthermore, it turns out that this distribution can be characterized by the distribution function 
\[V(x) = \begin{cases} 0 &\mbox{if } x \leq 0 \\
1-\frac{1}{\pi}\left(f\circ \rho^{-1}\right)(x) & \mbox{if } 0<x<\frac{(r+1)^{r+1}}{r^r}\\
1 &\mbox{if } x\geq \frac{(r+1)^{r+1}}{r^r},\end{cases} 
\]
where we explicitly have
\[\rho:\left(0, \frac{\pi}{r+1}\right)\rightarrow \left(0, \frac{(r+1)^{r+1}}{r^r}\right), \quad \rho(\varphi)=\frac{\left(\sin{(r+1)\varphi}\right)^{r+1}}{\sin{\varphi}\left(\sin{r\varphi}\right)^r},
\]
\[f:\left(0, \frac{\pi}{r+1}\right)\rightarrow \left(0, \pi\right), \quad f(\varphi)=(r+1)\varphi-r\frac{\sin{(r+1)\varphi}}{\sin{r\varphi}}\sin{\varphi}.
\]
Moreover, in Theorem \ref{C} we state an elementary and explicit description for the density of the Fuss-Catalan distribution of order \(r\) in the coordinates \(x=\rho(\varphi)\) by
\[v(\rho(\varphi))=\frac{(\sin\varphi)^2 (\sin r\varphi)^{r-1}}{\pi (\sin(r+1)\varphi)^{r}}.\]

\section{Auxiliary Results}

The first auxiliary result we mention is a simple version of the multivariate method of saddle points (see \cite{Neuschel} for a short proof and discussion).
\begin{theorem}\label{MSP}
Let \(p\) and \(q\) be holomorphic functions on a complex domain \(D\subset \mathbb{C}^r\) with \([-a,a]^r \subset D\) for a number \(a>0\), and let
\[I(n)=\int\limits_{[-a,a]^r} e^{-n p(t)} q(t) dt,\]
where \(t=(t_1,\ldots,t_r)\). Moreover, let \(t=0\) be a simple saddle point of the function \(p\), which means that we have for the complex gradient
\[\grad p (0) = 0\]
and for the Hessian
\[\det\Hess p (0) \neq 0.\]
Furthermore, suppose that, considered as a real-valued function on \([-a,a]^r\), \(\Re [p (w)]\) attains its minimum exactly at the point \(t=0\) with \(\det\Re\Hess p (0) \neq 0\) and \(q(0)\neq 0\). Then we have
\begin{equation}\label{A}I(n) = \left(\frac{2\pi}{n}\right)^{r/2} e^{-n p(0)} \frac{q(0)}{\sqrt{\det\Hess p (0)}} (1+o(1)),
\end{equation}
as \(n\rightarrow \infty\).
\end{theorem}

\begin{remark} The branch of the square root in (\ref{A}) is determined by the identity
\[\int\limits_{\mathbb{R}^r} e^{-\frac{1}{2} t^T\Hess p (0)t}dt=\frac{(2\pi)^{r/2}}{\sqrt{\det\Hess p (0)}}.\]
In general, the proper choice of the branch for the square root in (\ref{A}) can be described by fixing the arguments of the eigenvalues of \(\Hess p (0)\) in a correct manner (see, e.g., \cite{Fedoryuk}), which is related to the Maslov index.
\end{remark}

Next we prove a preliminary result on the location of the zeros of the polynomials in question.
\begin{lemma}\label{Zeros} Let \(r\in\mathbb{N}\) and \(\nu_1, \ldots, \nu_r \in \mathbb{N}_0\) be arbitrary integers, then all zeros of the polynomials \(F_n\) defined in (\ref{DefF}) are real and positive.
\end{lemma}
\begin{proof}Let \(\sum_{k=0}^n a_k x^k\) be any real polynomial such that all its roots are real. Then, of course, the same is true for the zeros of the polynomial \(\sum_{k=0}^n a_{n-k} x^k\). Moreover, let \(m\) be a nonnegative integer, then we conclude from \cite{Polya}, Part 5, Chapter 1, Ex. 63 (with \(P(x)=x^{n+m}\)) that all roots of the polynomial \(\sum_{k=0}^n \frac{a_k}{(k+m)!} x^k\) are real. Now, starting with the observation that all zeros of the polynomials \(\sum_{k=0}^n \binom{n}{k} (-x)^k\) are real, by mathematical induction with respect to \(r\) it can easily be established that all zeros of the polynomial \(F_n\) are real, thus they have to be positive.
\end{proof}
Finally, we provide an inequality for a function of several real variables which will be an important tool for the proof of the main result in Theorem \ref{PRA}.
\begin{lemma}\label{Inequality} Let \(r\in\mathbb{N}\), \(0<\varphi<\frac{\pi}{r+1}\), and let the function \(h:[-\pi,\pi]^r \rightarrow \mathbb{R}\) be defined by (\(\boldsymbol{x}=(x_1,\ldots,x_r)\))
\[h(\boldsymbol{x})=\exp{\left(2\frac{\sin{(r+1)\varphi}}{\sin{r\varphi}}\sum_{j=1}^{r}\cos{x_j}\right)}\left(\sin^2{\varphi}+\sin^2{(r+1)\varphi}-2\sin{\varphi}\sin{(r+1)\varphi}\cos{\Bigg(\sum_{j=1}^{r}x_j\Bigg)}\right).\]
Then the function \(h\) attains its global maximum exactly at the points
\[(x_1,\ldots,x_r)=(\varphi,\ldots,\varphi)\quad\text{and}\quad(x_1,\ldots,x_r)=(-\varphi,\ldots,-\varphi).\]
\end{lemma}
\begin{proof}We consider the proof only for \(r\geq3\) as in the case \(r\in\{1,2\}\) the statement can be established with less effort. First, let \(c \in (-\pi,\pi)\) be a real number and let the auxiliary function \(f_c:[-\pi,\pi]^{r-1}\rightarrow \mathbb{R}\) be defined by 
\[f_c(x_1,\ldots,x_{r-1})=\sum_{j=1}^{r-1} \cos{x_j}+\cos{\left(c-\sum_{j=1}^{r-1} x_j\right)}.\]
Then it is an elementary task of multivariate calculus to show that \(f_c\) takes its global maximum exactly at the point
\((x_1,\ldots,x_{r-1})=(\frac{c}{r},\ldots,\frac{c}{r})\). Moreover, if \(c\in\{-\pi,\pi\}\), then \(f_c\) attains its global maximum at the points \((x_1,\ldots,x_{r-1})=(\frac{c}{r},\ldots,\frac{c}{r})\) and \((x_1,\ldots,x_{r-1})=(-\frac{c}{r},\ldots,-\frac{c}{r})\). By extending the domain of the function \(f_c\) to \(\mathbb{R}^{r-1}\) and observing that \(f_c\) is \(2\pi\)-periodic with respect to all its variables, we can state more generally the following: If we have \(c\in\mathbb{R}\backslash\{(2l+1)\pi\vert l\in \mathbb{Z}\}\), then \(f_c\) attains its global maximum exactly in the points
\[(x_1,\ldots,x_{r-1})=\left(\frac{c-\left[\frac{c+\pi}{2\pi}\right] 2\pi}{r}+2\pi k_1,\ldots,\frac{c-\left[\frac{c+\pi}{2\pi}\right] 2\pi}{r}+2\pi k_{r-1}\right),\quad k_j \in \mathbb{Z},\]
where here and in the following \([\cdot]\) denotes the floor function. On the other hand, if we have \(c=(2l+1)\pi\) for an integer \(l\in\mathbb{Z}\), then the function \(f_c\) attains its global maximum exactly in the points
\[(x_1,\ldots,x_{r-1})=\left(-\frac{\pi}{r}+2\pi k_1,\ldots,-\frac{\pi}{r}+2\pi k_{r-1}\right),\quad k_j \in \mathbb{Z},\]
and
\[(x_1,\ldots,x_{r-1})=\left(\frac{\pi}{r}+2\pi k_1,\ldots,\frac{\pi}{r}+2\pi k_{r-1}\right),\quad k_j \in \mathbb{Z}.\]
Now, for a fixed \(c\in\mathbb{R}\), let us consider the function 
\[\tilde{g}(x_1,\ldots,x_r)=\sum_{j=1}^r\cos{x_j}\]
defined on the hyperplane \(H_c =\left\{(x_1,\ldots,x_r) \in \mathbb{R}^r\vert \sum_{j=1}^r x_j =c\right\}\). If we have \(c\in\mathbb{R}\backslash\{(2l+1)\pi\vert l\in \mathbb{Z}\}\), then the function \(\tilde{g}\) attains its global maximum exactly at the points (\(k_j \in \mathbb{Z}\)) 
\[\left(\frac{c-\left[\frac{c+\pi}{2\pi}\right] 2\pi}{r}+2\pi k_1,\ldots,\frac{c-\left[\frac{c+\pi}{2\pi}\right] 2\pi}{r}+2\pi k_{r-1}, \frac{c-\left[\frac{c+\pi}{2\pi}\right] 2\pi}{r}+\left[\frac{c+\pi}{2\pi}\right] 2\pi-2\pi\sum_{j=1}^{r-1}k_j\right).\]
On the other hand, if we have \(c=(2l+1)\pi\) for an integer \(l\in\mathbb{Z}\), then \(\tilde{g}\) attains its global maximum exactly at the points 
\[\left(-\frac{\pi}{r}+2\pi k_1,\ldots,-\frac{\pi}{r}+2\pi k_{r-1},-\frac{\pi}{r}+2\pi(l+1)-2\pi \sum_{j=1}^{r-1}k_j\right),\quad k_j \in \mathbb{Z},\]
and
\[\left(\frac{\pi}{r}+2\pi k_1,\ldots,\frac{\pi}{r}+2\pi k_{r-1},\frac{\pi}{r}+2\pi l-2\pi \sum_{j=1}^{r-1}k_j\right),\quad k_j \in \mathbb{Z}.\]
Considering the restriction of the function \(\tilde{g}\) from \(H_c\) to the set \(H_c \cap [-\pi,\pi]^r\), we can state the following: In the case \(c\notin[\pi,\pi]\) none of the above mentioned points at which \(\tilde{g}\) attains its global maximum (with respect to the domain \(H_c\)) is located in \(H_c \cap [-\pi,\pi]^r\). On the other hand, if we have \(c\in[\pi,\pi]\), then \(\tilde{g}\) attains its global maximum with respect to \(H_c \cap [-\pi,\pi]^r\) exactly at the point \((\frac{c}{r},\ldots,\frac{c}{r})\). This implies for \(c\in[-r\pi,r\pi]\backslash[-\pi,\pi]\)
\begin{equation}\label{L1}\max_{H_c \cap [-\pi,\pi]^r} \sum_{j=1}^r \cos{x_j} < \max_{H_{c-\left[\frac{c+\pi}{2\pi}\right] 2\pi} \cap [-\pi,\pi]^r} \sum_{j=1}^r \cos{x_j}.
\end{equation}
Now, in order to study the function \(h\) with respect to its maximum on \([-\pi,\pi]^r\), we will cut its domain into disjoint slices 
\[[-\pi,\pi]^r=\bigcup_{c\in[-r\pi,r\pi]} H_c \cap [-\pi,\pi]^r.\]
Using (\ref{L1}) and observing that the value of \(\cos\left(\sum_{j=1}^r x_j\right)\) remains the same as we move from the slice \(H_c \cap [-\pi,\pi]^r\) to the slice \(H_{c-\left[\frac{c+\pi}{2\pi}\right] 2\pi} \cap [-\pi,\pi]^r\), we can conclude 
\[\max_{c\in[-r\pi,r\pi]\backslash[-\pi,\pi]}~ \max_{H_c \cap [-\pi,\pi]^r} h(x_1,\ldots,x_r)<\max_{c\in[\pi,\pi]}~ \max_{H_c \cap [-\pi,\pi]^r} h(x_1,\ldots,x_r)\]
and thus
\begin{align*}&\max_{[-\pi,\pi]^r} h(x_1,\ldots,x_r)=\max_{c\in[-\pi,\pi]}~\max_{H_c \cap [-\pi,\pi]^r} h(x_1,\ldots,x_r) \\
&=\max_{c\in[-\pi,\pi]}\Bigg\{\exp{\left(2\frac{\sin{(r+1)\varphi}}{\sin{r\varphi}}\max_{H_c \cap [-\pi,\pi]^r}\sum_{j=1}^{r}\cos{x_j}\right)}\\
&~~\times\left(\sin^2{\varphi}+\sin^2{(r+1)\varphi}-2\sin{\varphi}\sin{(r+1)\varphi}\cos{c}\right)\Bigg\},
\end{align*}
where for every \(c\in[-\pi,\pi]\)
\[\max_{H_c \cap [-\pi,\pi]^r} \sum_{j=1}^{r}\cos{x_j}=r\cos{\frac{c}{r}}\]
is attained only in the point \((\frac{c}{r},\ldots,\frac{c}{r})\). Hence, we obtain
\begin{align}\label{L1.1}\max_{[-\pi,\pi]^r} h(x_1,\ldots,x_r)=\max_{c\in[-\pi,\pi]}\Bigg\{\exp{\left(2\frac{\sin{(r+1)\varphi}}{\sin{r\varphi}}r\cos{\frac{c}{r}}\right)}\\\nonumber
\times\left(\sin^2{\varphi}+\sin^2{(r+1)\varphi}-2\sin{\varphi}\sin{(r+1)\varphi}\cos{c}\right)\Bigg\}.
\end{align}
Next, let us determine the maxima of the function \(\tilde{h}:[-\pi,\pi]\rightarrow \mathbb{R}\) defined by
\[\tilde{h}(c)=\exp{\left(2\frac{\sin{(r+1)\varphi}}{\sin{r\varphi}}r\cos{\frac{c}{r}}\right)}\left(\sin^2{\varphi}+\sin^2{(r+1)\varphi}-2\sin{\varphi}\sin{(r+1)\varphi}\cos{c}\right).\]
As \(\tilde{h}\) is symmetric with respect to the origin, we can restrict our attention to the interval \([0,\pi]\). We have for its derivative
\begin{align*}\tilde{h}'(c)=&2\exp{\left(2\frac{\sin{(r+1)\varphi}}{\sin{r\varphi}}r\cos{\frac{c}{r}}\right)}\sin{(r+1)\varphi}~\sin{\varphi}~\sin{\frac{c}{r}}\\
&~\times\left\{\frac{\sin{c}}{\sin{\frac{c}{r}}}-\left(\frac{\sin^2{\varphi} +\sin^2{(r+1)\varphi}}{\sin{r\varphi}~\sin{\varphi}}-\frac{2\sin{(r+1)\varphi}}{\sin{r\varphi}}\cos{c}\right)\right\}.
\end{align*}
On the interval \((0,\pi)\) the function \(\frac{\sin{c}}{\sin{\frac{c}{r}}}\) is strictly decreasing from \(r\) to \(0\), whereas the function
\[\frac{\sin^2{\varphi} +\sin^2{(r+1)\varphi}}{\sin{r\varphi}~\sin{\varphi}}-\frac{2\sin{(r+1)\varphi}}{\sin{r\varphi}}\cos{c}\]
is strictly increasing from 
\[\frac{(\sin{\varphi}-\sin{(r+1)\varphi})^2}{\sin{\varphi}~\sin{r\varphi}}<r\quad \text{to}\quad \frac{(\sin{\varphi}+\sin{(r+1)\varphi})^2}{\sin{\varphi}~\sin{r\varphi}}.\]
Hence, there exists a unique \(c_0 \in (0,\pi)\) with \(\tilde{h}'(c_0)=0\). Now, by an elementary calculation we can verify \(c_0=r\varphi\), and we have \(\tilde{h}'(c)>0\) on the interval \((0,r\varphi)\) and \(\tilde{h}'(c)<0\) on the interval \((r\varphi,\pi)\). From this we conclude that the function \(\tilde{h}\) attains its global maximum on \([-\pi,\pi]\) exactly at the points \(c=r\varphi\) and \(c=-r\varphi\). Together with (\ref{L1.1}) this establishes the proof.
\end{proof}
\section{Main Results}
Now we turn to the first main result which describes the asymptotic behavior of the rescaled polynomials \(F_n (n^r x)\).
\begin{theorem}[Asymptotics of Plancherel-Rotach type]\label{PRA} Let \(r\in\mathbb{N}\) and \(\nu_1, \ldots, \nu_r \in \mathbb{N}_0\) be arbitrary integers, then for the polynomials \(F_n\) defined in (\ref{DefF}) we have 
\begin{align} \nonumber F_n (n^r x)=&\frac{2}{(2\pi)^{r/2}} \left(\frac{\sin{r\varphi}}{n \sin{(r+1)\varphi}}\right)^{\frac{r}{2}+\nu_1+\ldots+\nu_r} \exp\left\{nr\frac{\sin{(r+1)\varphi}}{\sin{r\varphi}} \cos{\varphi}\right\}\\ \label{PR2}
\times&\left(-\frac{\sin{r\varphi}}{\sin{\varphi}}\right)^n \left\{\left(1-\frac{r\sin{\varphi}\cos{(r+1)\varphi}}{\sin{r\varphi}}\right)^2+\left(\frac{r\sin{\varphi}\sin{(r+1)\varphi}}{\sin{r\varphi}}\right)^2\right\}^{-1/4}\\\nonumber
\times&\left\{\cos{\left(n\left(r\frac{\sin{(r+1)\varphi}}{\sin{r\varphi}}\sin{\varphi}-(r+1)\varphi\right)+g(r,\nu,\varphi)\right)}+o(1)\right\},
\end{align}
as \(n\rightarrow \infty\). In (\ref{PR2}) we use
\[x=\frac{\left(\sin{(r+1)\varphi}\right)^{r+1}}{\sin{\varphi}\left(\sin{r\varphi}\right)^r},\quad 0<\varphi<\frac{\pi}{r+1},\] 
parameterizing the interval \(\left(0,\frac{(r+1)^{r+1}}{r^r}\right)\), and the phase shift \(g(r,\nu,\varphi)\) is given by
\begin{equation}\label{Shift}g(r,\nu,\varphi)=-\left(\frac{r}{2}+\nu_1+\ldots+\nu_r\right)\varphi-\frac{1}{2}\arctan{\left(\frac{-r\sin{\varphi}\sin{(r+1)\varphi}}{\sin{r\varphi}-r\sin{\varphi \cos{(r+1)\varphi}}}\right)}.
\end{equation}
\end{theorem}
\begin{proof} At first we establish a representation for the polynomials \(F_n\) as a multivariate complex contour integral which is suitable for determining the asymptotic behavior via an application of the multivariate method of saddle points. To this end, using the binomial theorem and 
\[\frac{1}{2\pi i}\int\limits_{C} \exp{z}\frac{dz}{z^{N+1}} =\frac{1}{N!},\]
where \(N\in\mathbb{N}\) and \(C\) denotes a positive-oriented contour encircling the origin, it can easily be verified that we have
\[F_n (x)=\frac{1}{(2\pi i)^r}\int\limits_{\Gamma} \exp{\bigg(\sum_{j=1}^r w_j \bigg)}\left(1-\frac{x}{\prod_{j=1}^r w_j}\right)^n \frac{dw}{\prod_{j=1}^r w_j^{\nu_j +1}},\]
where \(w=(w_1,\ldots,w_r)\) and \(\Gamma\) is given as an \(r\)-fold product of positive-oriented contours in the complex plane encircling the origin. By the change of variables \(w_j \rightarrow nw_j\), for all \(j\in \{1,\dots,r\}\), we obtain
\begin{equation}\label{IR1}F_n (n^r x)=\frac{1}{n^{\nu_1+\ldots+\nu_r} (2\pi i)^r}\int\limits_{\Gamma} \Bigg\{\exp{\bigg(\sum_{j=1}^r w_j \bigg)}\left(1-\frac{x}{\prod_{j=1}^r w_j}\right) \Bigg\}^n\frac{dw}{\prod_{j=1}^r w_j^{\nu_j +1}}.
\end{equation}
Studying the multivariate saddle points of the function
\[\exp{\bigg(\sum_{j=1}^r w_j \bigg)}\left(1-\frac{x}{\prod_{j=1}^r w_j}\right)\]
shows that all saddle points are of the form \((w_1,\ldots,w_r)=(w,\ldots,w)\) where \(w\) satisfies the trinomial algebraic equation
\begin{equation}\label{SE}w^{r+1}-xw+x=0.
\end{equation}
In the case we are interested in, the number \(x>0\) will be located in an interval on which the polynomials \(F_n (n^r x)\) show an oscillatory behavior. Thus, we are expecting the essential contributions for the asymptotics coming from two particular saddle points which are complex conjugates. Having this in mind, introducing polar coordinates of the form \(w(\varphi)=a(\varphi)e^{i\varphi}\) and carefully studying the imaginary and the real part of equation (\ref{SE}) it turns out that using the parameterization 
\[x=\frac{\left(\sin{(r+1)\varphi}\right)^{r+1}}{\sin{\varphi}\left(\sin{r\varphi}\right)^r},\quad 0<\varphi<\frac{\pi}{r+1},\]
two roots of (\ref{SE}) are located at the points 
\[w=\frac{\sin{(r+1)\varphi}}{\sin{r\varphi}}e^{i\varphi} \quad\text{and}\quad w=\frac{\sin{(r+1)\varphi}}{\sin{r\varphi}}e^{-i\varphi}.\]
Thus, we are lead to use in (\ref{IR1}) the parameterizations (\(j\in\{1,\ldots,r\}\))
\[w_j=\frac{\sin{(r+1)\varphi}}{\sin{r\varphi}}e^{it_j},\quad t_j\in [-\pi,\pi].\]
This yields
\begin{align}\nonumber
&F_n (n^r x)=\frac{1}{(2\pi)^r} \left(\frac{\sin{r\varphi}}{n \sin{(r+1)\varphi}}\right)^{\nu_1+\ldots+\nu_r}\\\label{IR2}
~&\times\int\limits_{[-\pi,\pi]^r} \Bigg\{\exp{\bigg(\frac{\sin{(r+1)\varphi}}{ \sin{r\varphi}}\sum_{j=1}^r e^{it_j} \bigg)}\left(1-\frac{\sin{(r+1)\varphi}}{ \sin{\varphi}\prod_{j=1}^r e^{it_j}}\right)\Bigg\}^n \frac{dt}{\prod_{j=1}^r e^{i\nu_j t_j}}\\\nonumber
&\quad\quad=\frac{1}{(2\pi)^r} \left(\frac{\sin{r\varphi}}{n \sin{(r+1)\varphi}}\right)^{\nu_1+\ldots+\nu_r} \int\limits_{[-\pi,\pi]^r}e^{-n p(t)}q(t) dt, 
\end{align}
where \(t=(t_1,\ldots,t_r)\) and the functions \(p\) and \(q\) are defined by
\[p(t)=-\frac{\sin{(r+1)\varphi}}{ \sin{r\varphi}}\sum_{j=1}^r e^{it_j}-\log\left(1-\frac{\sin{(r+1)\varphi}}{ \sin{\varphi}\prod_{j=1}^r e^{it_j}}\right),\]
and
\[q(t)=\exp{\left\{-i\Bigg(\sum_{j=1}^r\nu_j t_j\Bigg)\right\}}.\]
An application of Lemma \ref{Inequality} shows that the function \(\vert e^{-p(t)}\vert = h(t)\) attains its global maximum exactly at the points \(t=(\varphi,\ldots,\varphi)\) and \(t=(-\varphi,\ldots,-\varphi)\). According to this, we split the integral into two parts and we obtain
\begin{equation}\label{Split}F_n (n^r x)=\frac{1}{(2\pi)^r} \left(\frac{\sin{r\varphi}}{n \sin{(r+1)\varphi}}\right)^{\nu_1+\ldots+\nu_r}\left\{I_n^{(1)}(\varphi)+I_n^{(2)}(\varphi)\right\},
\end{equation}
where
\[I_n^{(1)}(\varphi)=\int\limits_{E_1}e^{-n p(t)}q(t) dt\quad\text{and}\quad I_n^{(2)}(\varphi)=\int\limits_{E_2}e^{-n p(t)}q(t) dt\]
with
\[E_1 =\left\{x\in[-\pi,\pi]^r\Bigg\vert \sum_{j=1}^r x_j >0 \right\}\quad\text{and}\quad E_2 =\left\{x\in[-\pi,\pi]^r\Bigg\vert \sum_{j=1}^r x_j <0 \right\}.\]
Calculating the complex partial derivatives of \(p\) at the saddle points yields
\[\grad p (\varphi,\ldots,\varphi)=\grad p (-\varphi,\ldots,-\varphi)=0.\]
Moreover, we obtain
\begin{align*}
&\Hess p (\varphi,\ldots,\varphi) =\overline{\Hess p (-\varphi,\ldots,-\varphi)}\\
&=\frac{\sin{(r+1)\varphi}}{\sin{r\varphi}} e^{i\varphi} I_r - \frac{\sin{(r+1)\varphi}\, \sin{\varphi}}{\sin^2{r\varphi}} e^{i(r+2)\varphi} J_r,
\end{align*}
where \(I_r\) denotes the \(r\)-dimensional identity matrix and \(J_r=(1)_{1\leq j,k \leq r}\).
Hence, we can evaluate the determinants and we obtain
\[\det\Hess p (\varphi,\ldots,\varphi)=\overline{\det\Hess p (-\varphi,\ldots,-\varphi)}=\left(\frac{\sin{(r+1)\varphi}}{\sin{r\varphi}} e^{i\varphi}\right)^r \left(1-\frac{r \sin{\varphi}}{\sin{r\varphi}}e^{i(r+1)\varphi}\right).\]
Furthermore, we have for the determinants of the real parts of the Hessians
\begin{align*}
&\det \Re\Hess p (\varphi,\ldots,\varphi) =\det \Re\Hess p (-\varphi,\ldots,-\varphi)\\
&=\left(\frac{\sin{(r+1)\varphi}}{\sin{r\varphi}} \right)^r (\cos{\varphi})^{r-1} \left\{\cos{\varphi}-\frac{r \sin{\varphi \,\cos{(r+2)\varphi}}}{\sin{r\varphi}}\right\} >0.
\end{align*}
Now, by an application of Theorem \ref{MSP} (which, for simplicity, is formulated for saddle points located at the origin) we obtain
\[I_n^{(1)}(\varphi)=\left(\frac{2\pi}{n}\right)^{r/2} \frac{e^{-n p(\varphi,\ldots,\varphi)}q(\varphi,\ldots,\varphi)}{\left(\det\Hess p (\varphi,\ldots,\varphi)\right)^{1/2}}\left(1+o(1)\right),\]
as \(n\rightarrow \infty\). After an elementary calculation we arrive at
\begin{align*} I_n^{(1)}(\varphi)=&\left(\frac{2\pi}{n}\right)^{r/2}  \left(\frac{\sin{r\varphi}}{ \sin{(r+1)\varphi}}\right)^{\frac{r}{2}} \exp\left\{nr\frac{\sin{(r+1)\varphi}}{\sin{r\varphi}} \cos{\varphi}\right\}\\ 
\times&\left(-\frac{\sin{r\varphi}}{\sin{\varphi}}\right)^n \left\{\left(1-\frac{r\sin{\varphi}\cos{(r+1)\varphi}}{\sin{r\varphi}}\right)^2+\left(\frac{r\sin{\varphi}\sin{(r+1)\varphi}}{\sin{r\varphi}}\right)^2\right\}^{-1/4}\\
\times&\exp{\left\{i\left(n\left(r\frac{\sin{(r+1)\varphi}}{\sin{r\varphi}}\sin{\varphi}-(r+1)\varphi\right)+g(r,\nu,\varphi)\right)\right\}}(1+o(1)),
\end{align*}
as \(n\rightarrow\infty\), where the function \(g(r,\nu,\varphi)\) is defined in (\ref{Shift}). Let us write the latter expression in the form
\[I_n^{(1)}(\varphi)=G_n(\varphi)\exp\left\{i(n f(\varphi)+g(r,\nu,\varphi))\right\}(1+o(1)),\]
then we obtain in the same manner
\[I_n^{(2)}(\varphi)=G_n(\varphi)\exp\left\{-i(n f(\varphi)+g(r,\nu,\varphi))\right\}(1+o(1)),\]
as \(n\rightarrow\infty\).
Hence, we can conclude from (\ref{Split}) that we have
\[F_n (n^rx)=\frac{1}{(2\pi)^r} \left(\frac{\sin{r\varphi}}{n \sin{(r+1)\varphi}}\right)^{\nu_1+\ldots+\nu_r}2G_n(\varphi)\left\{\cos{\left(n f(\varphi)+g(r,\nu,\varphi)\right)}+o(1)\right\},\]
as \(n\rightarrow\infty\). From this statement the theorem follows.
\end{proof}

\begin{remark} As a special case of Theorem \ref{PRA} we obtain a formula of Plancherel-Rotach type for the classical Laguerre polynomials. For \(x=4n\cos^2{\varphi}\), \(0<\varphi<\frac{\pi}{2}\), we have
\[L_n^{(0)}(x)=(-1)^n \exp{\left\{2n\cos^2\varphi\right\}} (\pi n \sin{2\varphi})^{-1/2} \left\{\cos{\left(n(\sin{(2\varphi)}-2\varphi)-\varphi+\frac{\pi}{4}\right)} +o(1)\right\},\]
as \(n\rightarrow\infty\). This result resembles the classical result on those polynomials which uses a slightly different parameterization (see \cite{Szego}, p. 200): For \(x=(4n+2)\cos^2{\varphi}\), \(0<\varphi<\frac{\pi}{2}\), we have
\begin{align*}L_n^{(0)}(x)=&(-1)^n \exp{\left\{(2n+1)\cos^2\varphi\right\}} \left(\frac{\pi}{2} \sin{2\varphi}\right)^{-1/2} (2n(2n+1))^{-1/4}\\ &\times\left\{\sin{\left((n+1/2)(\sin{(2\varphi)}-2\varphi)+\frac{3\pi}{4}\right)} +o(1)\right\},
\end{align*}
as \(n\rightarrow\infty\).
\end{remark}
\begin{remark}\label{R2} As known from the case of the classical Laguerre polynomials \(L_n^{(0)}\), for instance, by a careful study of the remainder terms in the proofs it is possible to show that the asymptotic approximation in (\ref{PR2}) holds uniformly with respect to \(\epsilon \leq \varphi \leq \frac{\pi}{r+1}-\epsilon\) for arbitrary small \(\epsilon>0\). 
\end{remark}

For the purpose of illustration we take a look at a plot for the case \(r=3\), \(\nu_1 =2\), \(\nu_2=4\), \(\nu_3=5\) and \(n=150\). Showing the interval \(\left[0.5 \pi/4, 0.55\pi/4\right]\), in Fig. 1 the normalized polynomial \(\tilde{F}_n\) (solid line) and the associated cosine approximant \(c_n(\varphi)\) (dashed line) are plottet, where we have
\[\tilde{F}_n(\varphi) =\frac{F_n \left(n^r\frac{\left(\sin{(r+1)\varphi}\right)^{r+1}}{\sin{\varphi}\left(\sin{r\varphi}\right)^r}\right) \left\{\left(1-\frac{r\sin{\varphi}\cos{(r+1)\varphi}}{\sin{r\varphi}}\right)^2+\left(\frac{r\sin{\varphi}\sin{(r+1)\varphi}}{\sin{r\varphi}}\right)^2\right\}^{1/4}}{\frac{2}{(2\pi)^{r/2}} \left(\frac{\sin{r\varphi}}{n \sin{(r+1)\varphi}}\right)^{\frac{r}{2}+\nu_1+\ldots+\nu_r} \exp\left\{nr\frac{\sin{(r+1)\varphi}}{\sin{r\varphi}} \cos{\varphi}\right\}\left(-\frac{\sin{r\varphi}}{\sin{\varphi}}\right)^n },\]
and 
\[c_n(\varphi)=\cos{\left(n\left(r\frac{\sin{(r+1)\varphi}}{\sin{r\varphi}}\sin{\varphi}-(r+1)\varphi\right)+g(r,\nu,\varphi)\right)},\]
where the phase shift \(g(r,\nu,\varphi)\) is defined in (\ref{Shift}).

\begin{figure}[!htb]
\centering
\includegraphics[scale=0.5]{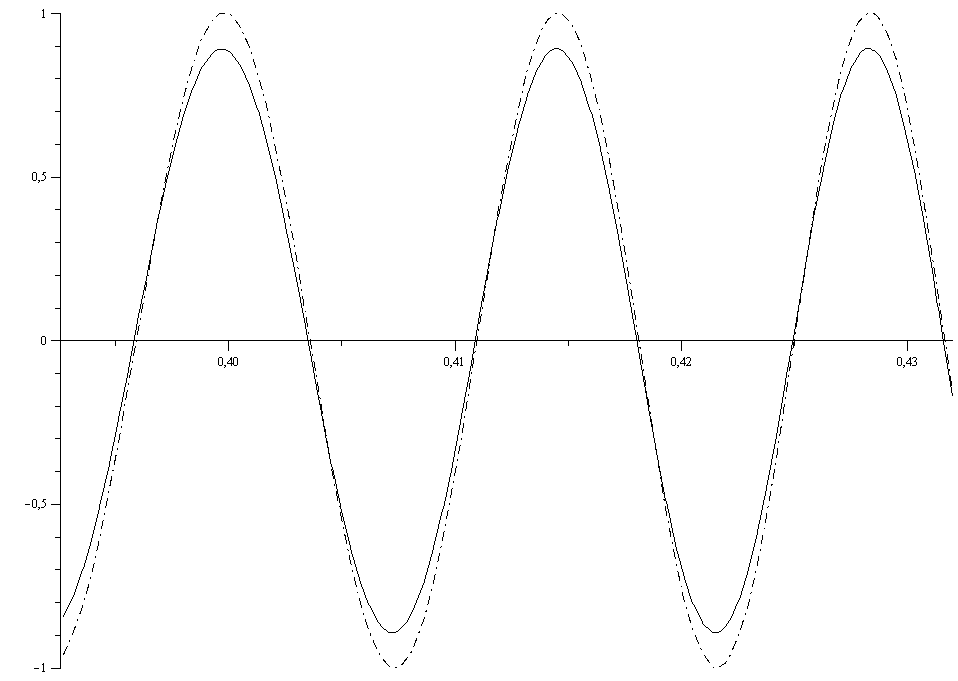}
\caption{The normalized polynomial \(\tilde{F}_n\) (solid line) and its cosine approximant \(c_n\) (dashed line)}
\label{Fig1}
\end{figure}

\bigskip

The next aim is to study the behavior of the zeros of the rescaled polynomials \(F_n (n^rx)\). Therefore, in regard to Theorem \ref{PRA}, let the functions \(\rho\) and \(f\) be defined by
\begin{equation}\label{rho}\rho:\left(0, \frac{\pi}{r+1}\right)\rightarrow \left(0, \frac{(r+1)^{r+1}}{r^r}\right), \quad \rho(\varphi)=\frac{\left(\sin{(r+1)\varphi}\right)^{r+1}}{\sin{\varphi}\left(\sin{r\varphi}\right)^r},
\end{equation}
\begin{equation}\label{f}f:\left(0, \frac{\pi}{r+1}\right)\rightarrow \left(0, \pi\right), \quad f(\varphi)=(r+1)\varphi-r\frac{\sin{(r+1)\varphi}}{\sin{r\varphi}}\sin{\varphi}.
\end{equation}
The function \(\rho\) is a strictly decreasing bijection, whereas the function \(f\) is a strictly increasing bijection. Hence, the composition \(f\circ \rho^{-1}\) is a strictly decreasing mapping from \(\left(0, \frac{(r+1)^{r+1}}{r^r}\right)\) onto \(\left(0, \pi\right)\), which admits a continuous extension of the same kind to the interval \(\left[0, \frac{(r+1)^{r+1}}{r^r}\right]\). Moreover, let the function \(V:\mathbb{R}\rightarrow [0,1]\) be defined by

\begin{equation}\label{V}V(x) = \begin{cases} 0 &\mbox{if } x \leq 0 \\
1-\frac{1}{\pi}\left(f\circ \rho^{-1}\right)(x) & \mbox{if } 0<x<\frac{(r+1)^{r+1}}{r^r}\\
1 &\mbox{if } x\geq \frac{(r+1)^{r+1}}{r^r}.\end{cases} 
\end{equation}
As it is not difficult to see that \(V\) is an increasing function on \(\mathbb{R}\) (with values in \([0, 1]\)) we can consider \(V\) as a probability distribution function.

\begin{theorem}[Asymptotic zero distribution]\label{WA}Let \(r\in\mathbb{N}\) and \(\nu_1, \ldots, \nu_r \in \mathbb{N}_0\) be arbitrary integers and let \((\mu_n)_n\) denote the sequence of normalized zero counting measures associated to the polynomials \(F_n (n^rx)\). Then the sequence \((\mu_n)_n\) converges in the weak-star sense to a unit measure \(\mu\) supported on \(\left[0, \frac{(r+1)^{r+1}}{r^r}\right]\) which is defined by the distribution function \(V\) in (\ref{V}). Moreover, the limit measure \(\mu\) coincides with the Fuss-Catalan distribution of order \(r\).
\end{theorem}
\begin{proof} Let \(\epsilon_1\) and \(\epsilon_2\) be real numbers with \(0<\epsilon_1 < \epsilon_2 < \frac{(r+1)^{r+1}}{r^r}\) and let us denote the zeros of \(F_n (n^rx)\) by \(x_{n1}\leq\ldots\leq x_{nn}\). At first we are interested in the behavior of 
\[\mu_n\left((\epsilon_1, \epsilon_2)\right)=\frac{1}{n} \left(\sum_{k=1}^n \delta_{x_{nk}}\right)((\epsilon_1, \epsilon_2))= \frac{1}{n} \left\vert\left\{k\in\{1,\ldots,n\}\vert x_{nk} \in (\epsilon_1, \epsilon_2)\right\}\right\vert\]
for large values of \(n\). Using the result (\ref{PR2}) together with Remark \ref{R2} we obtain following standard arguments
\begin{equation}\label{V1}\lim_{n\rightarrow \infty}\frac{1}{n} \left\vert\left\{k\in\{1,\ldots,n\}\vert x_{nk} \in (\epsilon_1, \epsilon_2)\right\}\right\vert =\frac{1}{\pi}f(\rho^{-1}(\epsilon_1))-\frac{1}{\pi}f(\rho^{-1}(\epsilon_2)),
\end{equation}
where the functions \(\rho\) and \(f\) are defined in (\ref{rho}), (\ref{f}). The arguments rely on the idea that for large values of \(n\) the result (\ref{PR2}) allows one to count the zeros of the cosine approximant instead of counting the zeros of the polynomial itself. Furthermore, we can observe
\[\lim_{\epsilon_1 \rightarrow 0} \frac{1}{\pi}f(\rho^{-1}(\epsilon_1)) =1\]
and
\[\lim_{\epsilon_2 \rightarrow \frac{(r+1)^{r+1}}{r^r}} \frac{1}{\pi}f(\rho^{-1}(\epsilon_2)) =0.\]
Now, by defining 
\begin{align*}\mu ((a,b))=&\frac{1}{\pi}f\left(\rho^{-1}\left(\min\left(\frac{(r+1)^{r+1}}{r^r},\max(a,0)\right)\right)\right)\\
&-\frac{1}{\pi}f\left(\rho^{-1}\left(\max\left(0,\min\left(b,\frac{(r+1)^{r+1}}{r^r}\right)\right)\right)\right)
\end{align*}
for arbitrary open intervals on the real axis we obtain a unit measure \(\mu\) on the borel sets with distribution function given by the function \(V\) in (\ref{V}). Moreover, it is not difficult to see that we have for all open intervals \((a,b)\subset\mathbb{R}\)
\[\lim_{n\rightarrow \infty}\frac{1}{n} \left(\sum_{k=1}^n \delta_{x_{nk}}\right)((a, b))=\mu ((a,b)).\]
In order to show that the measure \(\mu\) is the weak-star limit of the normalized zero counting measures \((\mu_n)_n\) it is sufficient to show that every subsequence of \((\mu_n)_n\) possesses a weak-star convergent subsequence with limit \(\mu\). Given an arbitrary subsequence of \((\mu_n)\), using Helly's selection principle (see, e.g., \cite{Saff/Totik}, p. 3) we can choose a weak-star convergent subsequence with limit measure \(\mu^{\ast}\), say. By an application of the Portmanteau theorem (see, e.g., \cite{Dudley}, p. 386) we obtain for every open interval 
\[\mu^{\ast}((a,b))\leq\mu((a,b)),\]
which, by virtue of the continuity of the distribution function \(V\), implies for every \(x\in\mathbb{R}\) 
\[\mu^{\ast}(\{x\})=0.\]
Hence, every open interval is a continuity set for \(\mu^{\ast}\) and again by an application of the Portmanteau theorem we have
\[\mu^{\ast}((a,b))=\mu((a,b)),\]
thus we obtain \(\mu^{\ast}=\mu\). 

Next, in order to show that \(\mu\) coincides with the Fuss-Catalan distribution we first take a look at the Stieltjes transform of the latter (abreviating \(x^{\ast} =\frac{(r+1)^{r+1}}{r^r}\))
\[F(z)=\int\limits_{0}^{x^{\ast}}\frac{v(x)}{z-x}dx,\]
where \(v(x)\) denotes its continuous density. Thus, we can consider \(F\) as a holomorphic function defined on \(\mathbb{C}\backslash[0,x^{\ast}]\). By expanding the integrand into a power series around infinity and using that the moments of the Fuss-Catalan distribution are given by the Fuss-Catalan numbers (see, e.g., \cite{Penson}) it is not difficult to see that we have for \(\vert z\vert > x^{\ast}\)
\begin{equation}\label{F} F(z)=\frac{1}{z}\sum_{n=0}^{\infty} \binom{rn+n}{n}\frac{1}{rn+1}\left(\frac{1}{z}\right)^n.
\end{equation}
Therefore, for instance from \cite{Polya1}, Part 3, Chapter 5, Ex. 211, we can conclude that the function \(w_1(z)=zF(z)\) satisfies the algebraic equation
\begin{equation}\label{AE} w(z)^{r+1}-zw(z)+z=0.
\end{equation}
Studying equation (\ref{AE}) yields that there are exactly two finite branch points which are located at \(z=0\) and \(z=x^{\ast}\) and for real \(z>x^{\ast}\) there are exactly two positive solutions of equation (\ref{AE}). One solution converges to unity as \(z\rightarrow +\infty\), so this one coincides with \(w_1(z)\), and the second positive solution behaves asymptotically like \(z^{1/r}\) as \(z\rightarrow +\infty\). Moreover, the associated sheets are connected to each other via the cut \((0,x^{\ast})\) and both branches converge to the value \(\frac{r+1}{r}\) as \(z\rightarrow x^{\ast}\), \(z>x^{\ast}\). From the proof of Theorem \ref{PRA} we remind the following: If we use the parameterization 
\[x=\rho(\varphi)=\frac{\left(\sin{(r+1)\varphi}\right)^{r+1}}{\sin{\varphi}\left(\sin{r\varphi}\right)^r},\quad 0<\varphi<\frac{\pi}{r+1},\]
in order to describe the values on the cut \(0<x<x^{\ast}\), two solutions of (\ref{AE}) are given by 
\begin{equation}\label{Roots}\frac{\sin{(r+1)\varphi}}{\sin{r\varphi}}e^{i\varphi} \quad\text{and}\quad \frac{\sin{(r+1)\varphi}}{\sin{r\varphi}}e^{-i\varphi},
\end{equation}
both converging to \(\frac{r+1}{r}\) as \(\varphi \rightarrow 0\), \(\varphi >0\). Hence, expressed by the new coordinate \(\varphi\), the functions in (\ref{Roots}) give the boundary values of \(w_1\) on the cut \((0,x^{\ast})\). Using the formula of Stieltjes-Perron we obtain for \(0<x<x^{\ast}\)
\begin{align*}v(x)&=\lim_{\epsilon \rightarrow 0+}\frac{1}{2\pi i} \left(F(x-i\epsilon)-F(x+i\epsilon)\right)\\
&=\lim_{\epsilon \rightarrow 0+}\frac{1}{2\pi i} \left(\frac{w_1(x-i\epsilon)}{x-i\epsilon}-\frac{w_1(x+i\epsilon)}{x+i\epsilon}\right).
\end{align*}
Now, putting \(x=\rho(\varphi)\) we have for \(0<\varphi<\frac{\pi}{r+1}\)
\begin{align}\nonumber v(\rho(\varphi))&=\frac{1}{\pi\rho(\varphi)}\frac{\sin(r+1)\varphi \sin\varphi}{\sin r\varphi}\\
&=\label{DEN}\frac{(\sin\varphi)^2 (\sin r\varphi)^{r-1}}{\pi (\sin(r+1)\varphi)^{r}}.
\end{align}
On the other hand, deriving the distribution function \(V\) we obtain the following expression for its density for \(0<\varphi<\frac{\pi}{r+1}\)
\begin{equation}\label{DV}V'(\rho(\varphi))=-\frac{1}{\pi}\frac{f'(\varphi)}{\rho'(\varphi)}.
\end{equation}
Observing that we have for the derivatives
\[f'(\varphi)=\frac{r^2 (\sin\varphi)^2+(r+1)(\sin r\varphi)^2-r\sin r\varphi \sin(r+2)\varphi}{(\sin r\varphi)^2},\]
\[\rho'(\varphi)=-\frac{r^2(\sin\varphi)^2 -2r\sin\varphi \sin r\varphi \cos(r+1)\varphi+(\sin r\varphi)^2}{(\sin\varphi)^2 (\sin r\varphi)^{r+1} (\sin(r+1)\varphi)^{-r}}\]
and using the fact that the numerators of both fractions coincide it can be verified that we have for \(0<\varphi<\frac{\pi}{r+1}\)
\begin{equation}\label{EQ}-\frac{1}{\pi}\frac{f'(\varphi)}{\rho'(\varphi)}=\frac{1}{\pi\rho(\varphi)}\frac{\sin(r+1)\varphi \sin\varphi}{\sin r\varphi},
\end{equation}
from which it follows that the density of \(V\) and that one of the Fuss-Catalan distribution coincide.
\end{proof}

\begin{remark} In general, the density function \(v\) of the Fuss-Catalan distribution of the order \(r\) can be described in terms of Meijer G-Functions (see \cite{Penson}) or in terms of multivariate integral representations (see \cite{Liu}). In the special cases \(r=1\) and \(r=2\) it is possible to derive more explicit representations. For instance, if \(r=1\), then the obtained asymptotic zero distribution can be described by the density
\[\frac{\sqrt{4-x}}{2\pi \sqrt{x}}\]
which is also known as the Marchenko-Pastur distribution, and it is well known to be the asymptotic zero distribution of the rescaled Laguerre polynomials (see, e.g., \cite{Gawronski}). 

However, the proof of Theorem \ref{WA} and especially the identity (\ref{DEN}) shows that after changing the coordinates from \(x\) to \(\varphi\) we obtain an elementary and explicit description for the density of the Fuss-Catalan distribution of order \(r\). We summarize this in the following theorem.
\end{remark}
\begin{theorem}[Fuss-Catalan distribution]\label{C} Let \(v(x)\) denote the continuous density of the Fuss-Catalan distribution of order \(r\) defined on \(\left(0, \frac{(r+1)^{r+1}}{r^r}\right)\). If
 \[x=\rho(\varphi)=\frac{\left(\sin{(r+1)\varphi}\right)^{r+1}}{\sin{\varphi}\left(\sin{r\varphi}\right)^r},\quad 0<\varphi<\frac{\pi}{r+1},\]
then we have
\[v(\rho(\varphi))=\frac{(\sin\varphi)^2 (\sin r\varphi)^{r-1}}{\pi (\sin(r+1)\varphi)^{r}}.\]
In the same sense we also obtain an explicit and elementary expression for the corresponding distribution function
\[V(\rho(\varphi))=1-\frac{(r+1)\varphi}{\pi}+r\frac{\sin{(r+1)\varphi}}{\pi\sin{r\varphi}}\sin{\varphi},\quad 0<\varphi<\frac{\pi}{r+1}.\] 
\end{theorem}

\begin{remark} Comparing the moments of the distribution \(V\) defined in (\ref{V}) with the moments of the Fuss-Catalan distribution, the Fuss-Catalan numbers, from Theorem \ref{WA} it follows
\[\int\limits_0^{\frac{(r+1)^{r+1}}{r^r}} x^n dV(x) =\frac{1}{\pi} \int\limits_0^{\frac{\pi}{r+1}} \rho(\varphi)^n f'(\varphi)d\varphi =\frac{1}{rn+1} \binom{rn+n}{n},\]
where the functions \(\rho\) and \(f\) are defined in (\ref{rho}) and (\ref{f}). Moreover, using the identity (\ref{EQ}) and integrating by parts we obtain the (remarkable) identity
\[\int\limits_{0}^{1} \frac{(\sin{\pi\varphi})^{(r+1)n}}{(\sin(\tfrac{\pi}{r+1}\varphi))^n (\sin(\tfrac{r\pi}{r+1}\varphi))^{rn}}d\varphi=\frac{r+1}{\pi}\int\limits_{0}^{\frac{\pi}{r+1}}\left(\frac{(\sin(r+1)\varphi)^{r+1}}{\sin\varphi (\sin r\varphi)^{r}}\right)^n d\varphi=\binom{(r+1)n}{n}.\]
\end{remark}

\section*{Acknowledgments} The author would like to express his deepest gratitude to Prof. Arno Kuijlaars as well as to Prof. Wolfgang Gawronski for providing most valuable advice.

\end{document}